\documentclass[12pt,a4paper]{article}
\usepackage{eurosym}
\usepackage{amsmath}
\usepackage{amssymb}
\usepackage{amsfonts}
\usepackage[T1]{fontenc}
\usepackage[latin1]{inputenc}
\usepackage{graphicx}

\newtheorem{theorem}{Theorem}

\newtheorem{corollary}[theorem]{Corollary}

\newtheorem{definition}[theorem]{Definition}
\newtheorem{example}[theorem]{Example}

\newtheorem{lemma}[theorem]{Lemma}

\newtheorem{proposition}[theorem]{Proposition}
\newtheorem{remark}[theorem]{Remark}

\newenvironment{proof}[1][Proof]{\noindent\textbf{#1.} }{\ \rule{0.5em}{0.5em}}

\begin{document}

\title{An alternative solvability criterion for the Dirichlet problem for
the minimal surface equation and an application to the mean curvature flow}
\author{A. Aiolfi, G. Nunes, J. Ripoll, L. Sauer, R. Soares }
\maketitle

\begin{abstract}
We propose an alternative condition for the solvability of the Diri-chlet
problem for the minimal surface equation that applies to non-mean convex
domains. We introduce a structural condition, obtained from a second-order
ordinary differential equation, which allows the con\nolinebreak struction
of explicit boundary barriers and it can also be applied to unbounded
domains. In the setting of Hadamard manifolds, this condition relates the
geometry of the domain to the admissible boun-dary data in a direct way.

In Euclidean space the condition leads to solvability under
geo-met\nolinebreak ric hypotheses of a different nature from those in the
classical Jenkins--Serrin theory, and in some configurations it applies
where the Jenkins--Serrin method does not. A central point of the present
approach is that the geometric restrictions and the boundary data enter
independently.

The same barrier construction can be used for graphical mean curvature flow.
This yields short-time existence with prescribed boun-dary values even when
the boundary is not mean convex. When mean convexity is present, one
recovers the classical graphical setting.
\end{abstract}

\section{Introduction}

Let $M$ be a complete $n$-dimensional Riemannian manifold, $n \geq 2$, and
let $\Omega \varsubsetneq M$ be a $C^2$-domain. The graph of a function $u
\in C^2(\Omega)$ is a minimal graph in $M \times \mathbb{R}$ if and only if $%
\mathcal{M}(u) = 0$ in $\Omega$, where 
\begin{equation}
\mathcal{M}(u) := \left(1 + |\nabla u|^2\right)\Delta u - \nabla^2 u(\nabla
u, \nabla u).  \label{Op}
\end{equation}

We consider the Dirichlet problem: 
\begin{align}
\mathcal{M}(u)& =0\quad \text{in }\Omega ,\quad u\in C^{2}(\overline{\Omega }%
),  \notag \\
u|_{\partial \Omega }& =\varphi |_{\partial \Omega },  \label{DP}
\end{align}%
where $\varphi \in C^{2}(\overline{\Omega })$ is given.

Let $H_{\partial \Omega}$ denote the mean curvature of $\partial \Omega$
with respect to the inner orientation, and define 
\begin{equation*}
\partial \Omega^{-} := \mathrm{clos}\left\{ p \in \partial \Omega \,;\,
H_{\partial \Omega}(p) < 0 \right\}.
\end{equation*}

In the case of bounded domains, it is well known that if $\Omega \subset 
\mathbb{R}^{n}$, the Dirichlet problem for the minimal surface equation
admits a solution for arbitrary continuous boundary data if and only if $%
\partial \Omega ^{-}=\emptyset $ (\cite{F}, \cite{JS}). The existence part
of this result has been generalized to Riemannian manifolds (see, e.g., \cite%
{DHL}). If $\partial \Omega ^{-}\neq \emptyset $, then, as shown in Theorem
1 of \cite{ARS}, which generalizes the classical result of H. Jenkins and J.
Serrin (Theorem 2 of \cite{JS}), a bound on the oscillation of the boundary
data $\varphi $ is sufficient to ensure solvability of \eqref{DP}, namely, 
\begin{equation}
\omega _{\varphi }(\partial \Omega ):=\sup_{\partial \Omega }\varphi
-\inf_{\partial \Omega }\varphi \leq B\left( |D\varphi |,|D^{2}\varphi |,|A|,%
\mathrm{Ric}_{M}\right) ,  \label{coc}
\end{equation}%
where $|A|$ is the norm of the second fundamental form of $\partial \Omega $%
, and $B$ is an explicit function (see \cite{ARS}, Section 2, p. 78 and
Lemma 4; see also \cite{JS}, Section 3, p. 179).

However, even in the Euclidean case, it is quite complicated to explicitly
verify whether a given $C^2$ boundary function satisfies the oscillation
bounds required in the above results, since the expression for $B$ depends
on a neighborhood of each point in $\partial \Omega^{-}$, which is hard to
describe explicitly. Furthermore, $B$ involves both the function and the
domain, making it difficult to treat them separately. To illustrate this, we
briefly explain the function $B$:

\paragraph{(I)}

In \cite{JS}, $B$ is constructed as follows. Since $\partial \Omega$ is
compact and of class $C^2$, there exists $l > 0$ (uniform for all $p \in
\partial \Omega$) such that the intersection $B_l(p) \cap \partial \Omega$
consists of a single connected component that is the graph over a domain $D$
in a hyperplane parallel to $T_p\partial \Omega$ of a function $g \in C^2(D)$
satisfying $|\nabla g| < 1$. Set $l^{\prime }= l/\sqrt{2}$. Also, since $%
\partial \Omega$ is compact and $C^2$, its principal curvatures are
uniformly bounded; let $k$ be an upper bound. Define 
\begin{align*}
\mathcal{A} &= \max \left\{ \frac{\pi}{l^{\prime }}, k, |D^2 \varphi|
\right\}, \\
\mathcal{C} &= \frac{\mathcal{A}}{(1 + |D\varphi|^2)^{8n}}, \\
\mathcal{H} &= \max_{\partial \Omega}(-H_{\partial \Omega}).
\end{align*}
Then $B$ is defined as 
\begin{equation}
B := \frac{1}{16n\mathcal{A}} \cdot 
\begin{cases}
\frac{\mathcal{C}}{\mathcal{H}}, & \text{if } \mathcal{C} \leq \mathcal{H},
\\ 
1 + \log\left( \frac{\mathcal{C}}{\mathcal{H}} \right), & \text{if } 0 < 
\mathcal{H} < \mathcal{C}, \\ 
+\infty, & \text{if } \mathcal{H} \leq 0.%
\end{cases}
\label{BE}
\end{equation}

\paragraph{(II)}

In \cite{ARS}, $B$ is defined through an expression involving $|A|$ and a
distance $\xi^{\prime }\in (0, d_0]$, where $d_0$ is such that the normal
exponential map 
\begin{equation}
\exp : \partial \Omega \times [0,d_0) \longrightarrow U_{d_0} = \{ z \in
\Omega \,;\, d(z) < d_0 \}  \label{exp}
\end{equation}
is a diffeomorphism (see Lemma 4 of \cite{ARS}). Here, $\xi^{\prime }$ plays
the same role as $l^{\prime }$ in the Jenkins-Serrin construction.

Given the above, we propose an alternative criterion to determine whether
the Dirichlet problem \eqref{DP} admits a solution for a given $\varphi \in
C^{2}(\overline{\Omega })$. Our goal is to retain the geometric insights of
previous results, while providing a more practical condition at least in the
Euclidean case that also applies to unbounded domains and, in the case of
Hadamard manifolds, can be phrased in a way that makes explicit the
relationship between the geometry of the domain and the behavior of the
boundary data that guarantees solvability (see Corollary~\ref{Cor1}). We are
inspired by \cite{RS}, which uses assumptions relating the $C^{2}$-norm of $%
\varphi $ and $\omega _{\varphi }(\overline{\Omega })$ with the maximal
radius of exterior tangent circles to $\partial \Omega $. Our criterion
arises from solving a second-order ODE (\ref{ODE}), which appears to be a
novel method for constructing barriers for the Dirichlet problem \eqref{DP}.

In special cases e.g., when $\Omega$ is bounded and the $C^2$ boundary data
lies between two horizontal slices it is possible to provide simpler
sufficient conditions involving $\omega_{\varphi}(\partial \Omega)$ to
ensure solvability (see \cite{AGLR}, \cite{ER}). Conditions on the
oscillation are also used for less regular boundary data, but they are
naturally more intricate (see \cite{W}, \cite{KT}, \cite{ANSS}, etc.).

\medskip

The ideas developed here extend naturally to the parabolic setting of
graphical mean curvature flow. Consider the evolution 
\begin{equation*}
\partial _{t}u=\mathcal{M}(u),
\end{equation*}%
where $\mathcal{M}$ is given by \eqref{Op}. The same differential inequality
underlying the ODE that generates our barrier reappears in the derivation of 
\emph{a priori} height and gradient controls for the parabolic problem. As a
consequence, the structural conditions ensuring solvability of the
stationary Dirichlet prob\nolinebreak lem provide, at the same time,
sufficient hypotheses to preserve graphicality and boundary regularity for
short time along the flow, even when $\partial \Omega $ is not mean convex
(see Section~\ref{SecFlow}). In the spirit of the work of Ecker-Huisken \cite%
{EH2} and within the Dirichlet setting of their subsequent work \cite{EH},
we establish short-time existence while dispensing with boundary mean
convexity. This yields a bridge between the elliptic theory and the mean
curvature flow of graphs and furnishes explicit sub/supersolutions usable at
the parabolic level (see, e.g., \cite{EH,AW}).

\medskip

In order to state our next results, we introduce some notation. Let $%
i:\nolinebreak M\rightarrow \lbrack 0,\infty ]$ denote the injectivity
radius function of $M$.

\begin{definition}
\label{egsc} We say that $\partial \Omega^{-}$ satisfies the \textbf{%
exterior sphere condition of radius} $r > 0$ if $\partial \Omega^{-} \neq
\emptyset$ and, at each $p \in \partial \Omega^{-}$, there exists a geodesic
sphere $\Sigma_p \subset M \setminus \Omega$ of radius $r$ tangent to $%
\partial \Omega$ at $p$, with $r < R$, where 
\begin{equation}
R := \inf \left\{ i(p_0) \,;\, p_0 \text{ is the center of } \Sigma_p,\, p
\in \partial \Omega^{-} \right\}.  \label{R}
\end{equation}
\end{definition}

This is abbreviated as the $r$-es condition.

\begin{definition}
We say that $\partial \Omega ^{-}$ satisfies the \textbf{locally Hadamard $r$%
-es condition} if it satisfies the $r$-es condition and, for each $p\in
\partial \Omega ^{-}$, letting $\ \Sigma _{p,s}:=d_{p}^{-1}(s)$ for $s\in
\lbrack 0,R-r]$ (with $\Sigma _{p,0}=\Sigma _{p}$ and $d_{p}=d_{M}(\cdot
,\Sigma _{p})$), and denoting $H_{\Sigma _{p,s}}$ the mean curvature of $%
\Sigma _{p,s}$ with respect to the outward orientation, we have: 
\begin{align}
H_{\Sigma _{p}}(p)& \leq H_{\Sigma _{p,s}}(x),  \label{HpHs} \\
|\nabla ^{2}d_{p}|(x)& \leq \max_{\Sigma _{p}}|\nabla ^{2}d_{p}|,
\label{HessC}
\end{align}%
for all $x\in \Sigma _{p,s}\cap \Omega $ and $s\in \lbrack 0,R-r]$.
\end{definition}

If $M$ is a Hadamard manifold and $\partial \Omega^{-}$ satisfies the $r$-es
condition, then it satisfies the locally Hadamard $r$-es condition with $R =
\infty$. The name reflects that the curvature behavior of the geodesic
spheres $\Sigma_{p,s}$ in this case mimics that of Hadamard manifolds.

Let us define: 
\begin{align}
\mu_r &:= \sup \left\{ \max_{\Sigma_p} |\nabla^2 d_p| \,;\, p \in \partial
\Omega^{-} \right\},  \label{mu2} \\
\lambda_r &:= \inf \left\{ H_{\Sigma_p}(p) \,;\, p \in \partial \Omega^{-}
\right\}.  \label{l2}
\end{align}
Note that $\lambda_r < 0$. For $M = \mathbb{R}^n$, we have $\lambda_r = -1/r$
and $\mu_r = 1/r$.

Let 
\begin{equation}
\tau = \tau(\varphi) := \max \left\{ \sup_{\overline{\Omega}} |\Delta
\varphi|,\ \sup_{\overline{\Omega}} |\nabla \varphi|,\ \sup_{\overline{\Omega%
}} |\nabla^2 \varphi| \right\},  \label{t}
\end{equation}
\begin{equation}
\varrho := \tau (1 + 2\tau^2),  \label{g}
\end{equation}
\begin{align}
a &:= \frac{2\tau^3 + 4\tau^2 + 3\tau}{\varrho},  \label{a} \\
b &:= -\frac{\tau^2 + 2\tau + 2}{\varrho},  \label{b} \\
c &:= \frac{\tau^2}{\varrho},  \label{c}
\end{align}
and denote $\omega := \omega_{\varphi}(\overline{\Omega})$.

If $\partial \Omega^{-}$ satisfies the $r$-es condition, define 
\begin{equation}
\theta := a + b(n-1)\lambda_r + c\mu_r.  \label{the}
\end{equation}
Note that $1 < a < \theta$ (the maximum possible value of $a$ is
approximately $3.7321$).

We now state the main result.

\begin{theorem}
\label{TA} Let $\Omega \subset M^n$ be a $C^2$-domain, where $M$ is a
complete Riemannian manifold, and let $\varphi \in C^2(\overline{\Omega})$
be given. Assume that $\tau < \infty$ if $\Omega$ is unbounded, where $\tau
= \tau(\varphi)$ is defined in \eqref{t}. If either:

\begin{itemize}
\item[(i)] $\partial \Omega^{-} = \emptyset$, or

\item[(ii)] $\partial \Omega^{-}$ satisfies the locally Hadamard $r$-es
condition, and 
\begin{equation}
\omega \leq \frac{1}{\varrho(\theta - 1)^2} \left[ \theta \left(1 -
e^{-\varrho(\theta - 1)\delta} \right) - \varrho(\theta - 1) \delta \right],
\label{mo}
\end{equation}
where $\varrho$ and $\theta$ are given by \eqref{g} and \eqref{the},
respectively, and 
\begin{equation}
\delta < \min \left\{ R - r,\ \frac{\ln \theta}{\varrho(\theta - 1)}
\right\},  \label{del}
\end{equation}
\end{itemize}

then the Dirichlet problem \eqref{DP} admits a solution. Moreover, the
solution is unique if $\Omega$ is bounded.
\end{theorem}

\begin{corollary}
\label{Cor1} Let $\Omega \subset M^{n}$ be a $C^{2}$-domain, where $M$ is a
Hadamard man\nolinebreak ifold, and let $\varphi \in C^{2}(\overline{\Omega }%
)$ be given. Assume that $\tau <\infty $ if $\Omega $ is unbounded, where $%
\tau =\tau (\varphi )$ is defined in \eqref{t}. If either:

\begin{itemize}
\item[(i)] $\partial \Omega^{-} = \emptyset$, or

\item[(ii)] $\partial \Omega^{-}$ satisfies the $r$-es condition, and 
\begin{equation}
\frac{ \sqrt{a - \ln a - 1} - \sqrt{\varrho \omega}(a + c\mu_r - 1) }{ \sqrt{%
\varrho \omega} \, b(n - 1)} \leq \lambda_r,  \label{SH}
\end{equation}
\end{itemize}

then the Dirichlet problem \eqref{DP} admits a solution. Moreover, the
solution is unique if $\Omega$ is bounded.
\end{corollary}

We remark that on the right-hand side of inequality \eqref{SH}, we have only
geometric information about the domain (specifically $\lambda_r,$ the
infimum of the mean curvature of the non mean convex part of the boundary of 
$\Omega$), while the left-hand side depends on the function $\varphi$ and
the ambient manifold $M$ (via $\mu_r$).

In the particular case where $M = \mathbb{R}^n$ and $\partial \Omega^{-}$
satisfies the $r$-es condition, we have $\lambda_r = -1/r$ and $\mu_r = 1/r$%
, so inequality \eqref{SH} becomes: 
\begin{equation}
0 < \frac{ \sqrt{\varrho \omega} \left( c - b(n - 1) \right) }{ \sqrt{a -
\ln a - 1} - (a - 1) \sqrt{\varrho \omega} } \leq r.  \label{SHE}
\end{equation}

Therefore, our condition is less restrictive than that of \cite{RS}, where $%
r\geq 1$ is required and, moreover, holds for arbitrary dimension. It also
allows us to conclude the solvability of the Dirichlet problem \eqref{DP} in
examples where the Jenkins-Serrin condition cannot be applied, as
illustrated at the end of Section~1.2.

Regarding the connection with the Jenkins--Serrin condition, we emphasize
that, even though our argument depends on an extension of $\varphi \in
C^{2}(\partial \Omega )$ into $\overline{\Omega }$, it is always possible to
construct such an extension---still denoted by $\varphi$---on a neighborhood 
$\Omega_{\varepsilon}\subset \Omega$ of $\partial \Omega$, for some $%
\varepsilon>0$, in such a way that 
\begin{equation*}
\sup_{\overline{\Omega}_{\varepsilon}}|\nabla \varphi| =\sup_{\partial
\Omega}|D\varphi|\quad\text{and}\quad \sup_{\overline{\Omega}%
_{\varepsilon}}|\nabla^2 \varphi| =\sup_{\partial \Omega}|D^2\varphi|,
\end{equation*}
where $D\varphi$ and $D^2\varphi$ denote the first and second tangential
derivatives of $\varphi$ on $\partial\Omega$ (see, for example, the
extension constructed in Section~2.1 of \cite{ARS}).

\subsection{Basic context}

Let $\left\langle \cdot,\cdot \right\rangle$ and $\nabla$ denote the
Riemannian metric and the Riemannian connection of $M$, respectively. Let $%
\Omega \subset M$ be a $C^2$ domain and suppose that $\partial \Omega^{-}$
satisfies the $r$-es condition. Fix a point $p \in \partial \Omega^{-}$ and
let $p_0$ be the center of the corresponding hypersurface $\Sigma_p$, as
given in Definition \ref{egsc}. Define 
\begin{equation*}
\Omega_R := B_R(p_0) \cap \Omega,
\end{equation*}
where $R$ is given by (\ref{R}). Consider the distance function 
\begin{equation*}
d(x) := d_p(x) = d_M(x,\Sigma_p) = d_M(x, p_0) - r,\quad x \in \Omega_R.
\end{equation*}

Since $\nabla d$ is orthogonal to the level sets $\Sigma_{p,s} := d^{-1}(s)$%
, for $s = d(x)$, with $\Sigma_p = \Sigma_0$ and $s \in [0, R - r)$, we can
define an orthonormal frame in a neighborhood of any given point $x \in
\Omega_R$ as follows. First, near $x \in \Sigma_{p,s} \cap \Omega_R$, we
choose an orthonormal frame $\{E_i\}_{i=1}^{n-1}$ tangent to $\Sigma_{p,s}$.
We then extend this frame to a neighborhood $\Lambda \subset \Omega_R$ of $x$
by parallel transport along the geodesic rays emanating from $p_0$, and we
continue to denote the extension by $\{E_i\}_{i=1}^{n-1}$. Define $E_n :=
\nabla d$ on $\Lambda$. Then, $\{E_i\}_{i=1}^n$ forms an orthonormal frame
on $\Lambda \subset \Omega_R$.\footnote{%
Without loss of generality, we may assume from the outset that $\Lambda =
B_R(p_0) \cap \Omega$, i.e., $\Lambda = \Omega_R$.}

Observe that, since the vector fields $E_i$ are parallel along the geodesics
issuing from $p_0$, we have $\nabla_{E_n} E_i = 0$ on $\Lambda$ for all $i =
1, \dots, n$.

\begin{lemma}
\label{Lb} In a neighborhood $\Lambda \subset \Omega_R$ as described above,
we have: 
\begin{equation*}
\nabla^2 d(E_i, E_n) = 0,\quad \text{for } i = 1, \dots, n,
\end{equation*}
and 
\begin{equation*}
\nabla^2 d(E_i, E_j) = -H_{E_n}(E_i, E_j),\quad \text{for } i,j = 1, \dots,
n-1,
\end{equation*}
where 
\begin{equation*}
H_{E_n}(E_i, E_j) := \left\langle -(\nabla_{E_i} E_n)^{\top}, E_j
\right\rangle.
\end{equation*}
In particular, $\nabla^2 d(E_i, E_i) = -II_{E_n}(E_i)$, where $II$ denotes
the second fundamental form of the hypersurfaces $\Sigma_{p,s}$, and 
\begin{equation*}
\Delta d = -\sum_{i=1}^{n-1} II_{E_n}(E_i).
\end{equation*}
\end{lemma}

\begin{proof}
Let $E_n=\nabla d$ and $\{E_1,\dots,E_{n-1}\}$ tangent to the level set $%
\Sigma_{p,s}$, extended off $\Sigma_{p,s}$ by parallel transport along the
radial geodesics from $p_0$. Since $|\nabla d|=1$ on the regular set of $d$,
we have 
\begin{equation*}
0=\tfrac12 E_i\!\left(|\nabla d|^2\right)=\langle \nabla_{E_i}\nabla
d,\nabla d\rangle =\nabla^2 d(E_i,E_n)
\end{equation*}
for every $i=1,\dots,n$, proving the first assertion.

For $i,j\leq n-1$, $\nabla ^{2}d(E_{i},E_{j})=\langle \nabla
_{E_{i}}E_{n},E_{j}\rangle $. Decompose $\nabla _{E_{i}}E_{n}$ into
tangential and normal parts relative to $\Sigma _{p,s}$: $\nabla
_{E_{i}}E_{n}=(\nabla _{E_{i}}E_{n})^{\top }+(\nabla _{E_{i}}E_{n})^{\perp }$%
. But $E_{n}$ is the unit normal to $\Sigma _{p,s}$, thus the Weingarten map 
$A_{E_{n}}$ is defined by $A_{E_{n}}(E_{i})=-(\nabla _{E_{i}}E_{n})^{\top }$
and the second fundamental form by $II_{E_{n}}(E_{i},E_{j})=\langle
A_{E_{n}}(E_{i}),E_{j}\rangle $. Hence 
\begin{equation*}
\nabla ^{2}d(E_{i},E_{j})=\langle \nabla _{E_{i}}E_{n},E_{j}\rangle
=\left\langle (\nabla _{E_{i}}E_{n})^{\top },E_{j}\right\rangle
=-\,II_{E_{n}}(E_{i},E_{j}),
\end{equation*}%
which yields the stated identity with $H_{E_{n}}(E_{i},E_{j}):=\langle
-(\nabla _{E_{i}}E_{n})^{\top },E_{j}\rangle $. Tak\nolinebreak ing the
trace over the tangential directions, 
\begin{equation*}
\Delta d=\sum_{k=1}^{n}\nabla ^{2}d(E_{k},E_{k})=\sum_{i=1}^{n-1}\nabla
^{2}d(E_{i},E_{i})+\nabla
^{2}d(E_{n},E_{n})=-\sum_{i=1}^{n-1}II_{E_{n}}(E_{i}),
\end{equation*}%
since $\nabla ^{2}d(E_{n},E_{n})=\langle \nabla _{E_{n}}E_{n},E_{n}\rangle
=0 $ by $|\nabla d|=1$ and metric-compatibility.
\end{proof}

\begin{lemma}
\label{LbHw} Let $\varphi \in C^2(\overline{\Omega})$ and $\psi \in
C^\infty([0, \infty))$ be given. Define the function $w \in C^2(\overline{%
\Lambda})$ by 
\begin{equation}
w = \varphi + \psi \circ d,  \label{w}
\end{equation}
where $\Lambda \subset \Omega_R$ is a neighborhood as defined above. Then, 
\begin{align}
\nabla^2 w(\nabla w, \nabla w) &= \left| \nabla \varphi \right|^2 \nabla^2
\varphi \left( \frac{\nabla \varphi}{\left| \nabla \varphi \right|}, \frac{%
\nabla \varphi}{\left| \nabla \varphi \right|} \right)  \notag \\
&\quad + \psi^{\prime }(d) \left[ 2 \left| \nabla \varphi \right| \nabla^2
\varphi \left( E_n, \frac{\nabla \varphi}{\left| \nabla \varphi \right|}
\right) + \left| \nabla \varphi \right|^2 \nabla^2 d \left( \frac{\nabla
\varphi}{\left| \nabla \varphi \right|}, \frac{\nabla \varphi}{\left| \nabla
\varphi \right|} \right) \right]  \notag \\
&\quad + \left[\psi^{\prime }(d)\right]^2 \nabla^2 \varphi(E_n, E_n) +
\psi^{\prime \prime }(d) \left( E_n(\varphi) + \psi^{\prime }(d) \right)^2.
\label{Hw}
\end{align}
\end{lemma}

\begin{proof}
We compute $\nabla _{\nabla w}\nabla w$ using linearity and the chain rule.
Since $w=\varphi +\psi (d)$ and $\nabla (\psi \circ d)=\psi ^{\prime
}(d)\nabla d=\psi ^{\prime }(d)E_{n}$, 
\begin{equation*}
\nabla _{\nabla w}\nabla w=\nabla _{\nabla \varphi }\nabla \varphi +\psi
^{\prime }(d)\left[ \nabla _{\nabla \varphi }E_{n}+\nabla _{E_{n}}\nabla
\varphi +\psi ^{\prime \prime }(d)E_{n}\right] +\psi ^{\prime \prime
}(d)E_{n}(\varphi )E_{n}.
\end{equation*}%
Taking the inner product with $\nabla w=\nabla \varphi +\psi ^{\prime
}(d)E_{n}$ and noting $\langle \nabla _{\nabla \varphi }E_{n},E_{n}\rangle
=\nolinebreak 0$ (since $\langle E_{n},E_{n}\rangle =1$), we obtain 
\begin{align*}
\langle \nabla _{\nabla w}\nabla w,\nabla w\rangle & =\nabla ^{2}\varphi
(\nabla \varphi ,\nabla \varphi )+\psi ^{\prime }(d)\big[2\,\nabla
^{2}\varphi (E_{n},\nabla \varphi )+\nabla ^{2}d(\nabla \varphi ,\nabla
\varphi )\big] \\
& \quad +\big[\psi ^{\prime }(d)\big]^{2}\nabla ^{2}\varphi
(E_{n},E_{n})+\psi ^{\prime \prime }(d)\big(E_{n}(\varphi )+\psi ^{\prime
}(d)\big)^{2},
\end{align*}%
which is \eqref{Hw} after factoring $|\nabla \varphi |$ where appropriate.
\end{proof}

\begin{proposition}
\label{pb} Assume that $\partial \Omega^{-}$ satisfies the locally Hadamard $%
r$-es condition. Let $\varphi \in C^{2}\left( \overline{\Omega} \right)$ be
given, and let $\sigma := \varrho \left( 1 - \theta \right)$, where $\varrho$
and $\theta$ are given by (\ref{g}) and (\ref{the}), respectively. Consider
the function 
\begin{equation}
\psi(s) = \frac{1}{\sigma} \left[ \varrho s + \left( \frac{\sigma - \varrho}{%
\sigma} \right) \left( e^{\sigma s} - 1 \right) \right], \quad s \in [0,
\delta] ,  \label{pisi}
\end{equation}
where 
\begin{equation*}
\delta < \min \left\{ R - r, \frac{1}{\sigma} \ln \left( \frac{\varrho}{%
\varrho - \sigma} \right) \right\},
\end{equation*}
and set $\Lambda_{p,\delta} = B_{r+\delta}(p_{0}) \cap \Omega$. Then the
function $w = \varphi + \psi \circ d$ belongs to $C^{2}(\overline{\Lambda}%
_{p,\delta})$ and satisfies 
\begin{equation*}
w(p) = \varphi(p), \quad w(x) \geq \varphi(x), \quad \text{and} \quad 
\mathcal{M}(w(x)) \leq 0 \quad \text{for all } x \in \overline{\Lambda}%
_{p,\delta}.
\end{equation*}
\end{proposition}

\begin{proof}
We explain how the function (\ref{pisi}) was constructed. We start by
assuming only that $\psi \in C^{\infty }\left( [0,\infty \right) $ and
satisfies the conditions $\psi \left( 0\right) =0$, $\psi ^{\prime }\left(
s\right) >0$ and $\psi ^{\prime \prime }\left( s\right) <0$ for $s>0$,
necessary conditions for $\psi $ to serve as an upper barrier. We assume,
from now onward, that $\Lambda \subset \Omega _{R}$ is such that $p\in
\partial \Lambda $. From (\ref{Op}), we have $\mathcal{M}(w)(x)\leq 0$ in $%
\Lambda $ if 
\begin{equation*}
(1+|\nabla w|^{2})\Delta w(x)-\nabla ^{2}w(\nabla w,\nabla w)(x)\leq 0,
\end{equation*}%
where $w=\varphi +\psi \circ d$. By Lemma \ref{LbHw}, $\nabla ^{2}w(\nabla
w,\nabla w)$ can be expressed as in (\ref{Hw}). Since%
\begin{equation*}
\left( 1+\left\vert \nabla w\right\vert ^{2}\right) \Delta w=\left(
1+\left\vert \nabla w\right\vert ^{2}\right) \Delta \varphi +(1+\left\vert
\nabla w\right\vert ^{2})\psi ^{\prime \prime }\left( d\right) +\left(
1+\left\vert \nabla w\right\vert ^{2}\right) \psi ^{\prime }\left( d\right)
\Delta d,
\end{equation*}%
it follows that $\mathcal{M}(w)(x)\leq 0$ if the following sum is
nonpositive:%
\begin{eqnarray*}
&&\left( 1+\left\vert \nabla w\right\vert ^{2}\right) \Delta \varphi
+(1+\left\vert \nabla w\right\vert ^{2})\psi ^{\prime \prime }\left(
d\right) +\left( 1+\left\vert \nabla w\right\vert ^{2}\right) \psi ^{\prime
}\left( d\right) \Delta d+ \\
&&-\left\vert \nabla \varphi \right\vert ^{2}\nabla ^{2}\varphi \left( \frac{%
\nabla \varphi }{\left\vert \nabla \varphi \right\vert },\frac{\nabla
\varphi }{\left\vert \nabla \varphi \right\vert }\right) + \\
&&-2\psi ^{\prime }\left( d\right) \left\vert \nabla \varphi \right\vert
\nabla ^{2}\varphi \left( E_{n},\frac{\nabla \varphi }{\left\vert \nabla
\varphi \right\vert }\right) -\psi ^{\prime }\left( d\right) \left\vert
\nabla \varphi \right\vert ^{2}\nabla ^{2}d\left( \frac{\nabla \varphi }{%
\left\vert \nabla \varphi \right\vert },\frac{\nabla \varphi }{\left\vert
\nabla \varphi \right\vert }\right) + \\
&&-\left[ \psi ^{\prime }\left( d\right) \right] ^{2}\nabla ^{2}\varphi
(E_{n},E_{n})-\psi ^{\prime \prime }\left( d\right) \left( E_{n}\left(
\varphi \right) +\psi ^{\prime }\left( d\right) \right) ^{2}.
\end{eqnarray*}%
In $\Lambda $, we estimate: 
\begin{align*}
-|\nabla \varphi |^{2}\nabla ^{2}\varphi \left( \frac{\nabla \varphi }{%
|\nabla \varphi |},\frac{\nabla \varphi }{|\nabla \varphi |}\right) & \leq
\tau ^{3}, \\
-|\nabla \varphi |\nabla ^{2}\varphi \left( E_{n},\frac{\nabla \varphi }{%
|\nabla \varphi |}\right) & \leq \tau ^{2}, \\
-|\nabla \varphi |^{2}\nabla ^{2}d\left( \frac{\nabla \varphi }{|\nabla
\varphi |},\frac{\nabla \varphi }{|\nabla \varphi |}\right) & \leq \tau
^{2}\mu _{r}, \\
-\nabla ^{2}\varphi (E_{n},E_{n})& \leq \tau .
\end{align*}%
where $\tau $ is given by (\ref{t}). The third inequality uses that $\nabla
^{2}d_{p}|_{\Sigma _{p,s}}\leq \mu _{r}$ for all $s\in \lbrack 0,R]$, due to
the locally Hadamard $r$-es condition. Therefore, the sum is nonpositive in $%
\Lambda $ if 
\begin{eqnarray}
&&\left( 1+\left\vert \nabla w\right\vert ^{2}\right) \tau +(1+\left\vert
\nabla w\right\vert ^{2})\psi ^{\prime \prime }\left( d\right) +\left(
1+\left\vert \nabla w\right\vert ^{2}\right) \psi ^{\prime }\left( d\right)
\Delta d+  \notag \\
&&+\tau ^{3}+2\psi ^{\prime }\left( d\right) \tau ^{2}+\psi ^{\prime }\left(
d\right) \tau ^{2}\mu _{r}  \notag \\
&&+\left[ \psi ^{\prime }\left( d\right) \right] ^{2}\tau -\psi ^{\prime
\prime }\left( d\right) \left( E_{n}\left( \varphi \right) +\psi ^{\prime
}\left( d\right) \right) ^{2}  \label{i1}
\end{eqnarray}%
is non positive. Now compute%
\begin{eqnarray*}
\left( 1+\left\vert \nabla w\right\vert ^{2}\right) \psi ^{\prime \prime
}\left( d\right) &=&\psi ^{\prime \prime }\left( d\right) +\left\vert \nabla
\varphi \right\vert ^{2}\psi ^{\prime \prime }\left( d\right) + \\
&&+2\psi ^{\prime }\left( d\right) E_{n}\left( \varphi \right) \psi ^{\prime
\prime }\left( d\right) +\left[ \psi ^{\prime }\left( d\right) \right]
^{2}\psi ^{\prime \prime }\left( d\right)
\end{eqnarray*}%
and, since $\psi ^{\prime \prime }<0$, we obtain%
\begin{equation*}
\left( 1+\left\vert \nabla w\right\vert ^{2}\right) \psi ^{\prime \prime
}\left( d\right) -\psi ^{\prime \prime }\left( d\right) \left( E_{n}\left(
\varphi \right) +\psi ^{\prime }\left( d\right) \right) ^{2}\leq \psi
^{\prime \prime }\left( d\right) .
\end{equation*}%
Then (\ref{i1}) is non positive in $\Lambda $ if%
\begin{eqnarray}
&&\left( 1+\left\vert \nabla w\right\vert ^{2}\right) \tau +\left(
1+\left\vert \nabla w\right\vert ^{2}\right) \psi ^{\prime }\left( d\right)
\Delta d+  \notag \\
&&\tau ^{3}+2\psi ^{\prime }\left( d\right) \tau ^{2}+\psi ^{\prime }\left(
d\right) \tau ^{2}\mu _{r}  \notag \\
&&+\left[ \psi ^{\prime }\left( d\right) \right] ^{2}\tau +\psi ^{\prime
\prime }\left( d\right) \leq 0.  \label{i2}
\end{eqnarray}%
We have%
\begin{eqnarray*}
\left( 1+\left\vert \nabla w\right\vert ^{2}\right) \tau &\leq &\left(
1+\left\vert \nabla \varphi \right\vert ^{2}+2\psi ^{\prime }\left( d\right)
\left\vert \nabla \varphi \right\vert +\left[ \psi ^{\prime }\left( d\right) %
\right] ^{2}\right) \tau \\
&\leq &\tau +\tau ^{3}+2\tau ^{2}\psi ^{\prime }\left( d\right) +\tau \left[
\psi ^{\prime }\left( d\right) \right] ^{2}.
\end{eqnarray*}%
Thus, inequality (\ref{i2}) holds if%
\begin{eqnarray}
&&\tau +2\tau ^{3}+4\tau ^{2}\psi ^{\prime }\left( d\right) +2\tau \left[
\psi ^{\prime }\left( d\right) \right] ^{2}+\left( 1+\left\vert \nabla
w\right\vert ^{2}\right) \psi ^{\prime }\left( d\right) \Delta d+  \notag \\
&&+\psi ^{\prime }\left( d\right) \tau ^{2}\mu _{r}+\psi ^{\prime \prime
}\left( d\right)  \label{i3}
\end{eqnarray}%
is nonpositive. Using Lemma \ref{Lb}, for $x\in \Lambda $ and $s=d(x)\in
\lbrack 0,R)$, we have 
\begin{equation*}
\Delta d(x)=-\sum_{i=1}^{n-1}II_{E_{n}}(E_{i})(x)=-(n-1)H_{\Sigma _{p,s}}(x),
\end{equation*}%
where $H_{\Sigma _{p,s}}$ is the mean curvature of the immersion $\Sigma
_{p,s}$ with respect to $E_{n}=\nabla d$. Since $p\in \partial \Omega ^{-}$,
we have 
\begin{equation*}
H_{\Sigma _{p}}(p)\leq H_{\partial \Omega }(p)<0,
\end{equation*}%
and from the locally Hadamard $r$-es condition, 
\begin{equation*}
\lambda _{r}\leq H_{\Sigma _{p}}(p)\leq H_{\Sigma _{p,s}}(x),
\end{equation*}%
for all $x\in \Sigma _{p,s}\cap \Omega _{R}$ and $s\in \lbrack 0,R-r]$.
Hence, 
\begin{equation*}
\Delta d(x)\leq -(n-1)\lambda _{r}
\end{equation*}%
and we have (\ref{i3}) if 
\begin{eqnarray}
&&\tau +2\tau ^{3}+4\tau ^{2}\psi ^{\prime }\left( d\right) +2\tau \left[
\psi ^{\prime }\left( d\right) \right] ^{2}-\left( 1+\left\vert \nabla
w\right\vert ^{2}\right) \psi ^{\prime }\left( d\right) \left( n-1\right)
\lambda _{r}+  \notag \\
&&+\psi ^{\prime }\left( d\right) \tau ^{2}\mu _{r}+\psi ^{\prime \prime
}\left( d\right) \leq 0.  \label{ing1}
\end{eqnarray}%
Observing that%
\begin{equation*}
\left( 1+\left\vert \nabla w\right\vert ^{2}\right) \psi ^{\prime }\left(
d\right) \leq \left( 1+\tau ^{2}\right) \psi ^{\prime }\left( d\right) +2 
\left[ \psi ^{\prime }\left( d\right) \right] ^{2}\tau \mathfrak{+}\left[
\psi ^{\prime }\left( d\right) \right] ^{3}
\end{equation*}%
it follows that we have (\ref{ing1}) in $\Lambda $ if%
\begin{eqnarray}
\psi ^{\prime \prime }\left( d\right) &\leq &\left( n-1\right) \lambda _{r} 
\left[ \psi ^{\prime }\left( d\right) \right] ^{3}-2\tau \left[ 1-\left(
n-1\right) \lambda _{r}\right] \left[ \psi ^{\prime }\left( d\right) \right]
^{2}  \notag \\
&&-\left[ \left( 4+\mu _{r}\right) \tau ^{2}-\left( 1+\tau ^{2}\right)
\left( n-1\right) \lambda _{r}\right] \psi ^{\prime }\left( d\right)  \notag
\\
&&-\tau \left[ 1+2\tau ^{2}\right] .  \notag
\end{eqnarray}%
Letting 
\begin{equation*}
\alpha :=2\tau \lbrack 1-(n-1)\lambda _{r}],\quad \beta :=(4+\mu _{r})\tau
^{2}-(1+\tau ^{2})(n-1)\lambda _{r},
\end{equation*}%
we write the inequality above as 
\begin{equation}
\psi ^{\prime \prime }\leq \left( n-1\right) \lambda _{r}\left[ \psi
^{\prime }\right] ^{3}-\alpha \left[ \psi ^{\prime }\right] ^{2}-\beta \psi
^{\prime }-\varrho .  \label{ing2}
\end{equation}%
Assume $0<\psi ^{\prime }\leq 1$. Then%
\begin{eqnarray*}
\left( n-1\right) \lambda _{r}\psi ^{\prime } &\leq &\left( n-1\right)
\lambda _{r}\left[ \psi ^{\prime }\right] ^{3}<0 \\
-\alpha \psi ^{\prime } &\leq &-\alpha \left[ \psi ^{\prime }\right] ^{2}<0
\end{eqnarray*}%
so inequality \eqref{ing2} is satisfied if $\psi ^{\prime \prime }-\sigma
\psi ^{\prime }+\varrho \leq 0$, with $\sigma =\varrho (1-\theta )<0$ since $%
\theta >1$. Thus, we solve the ODE 
\begin{equation}
\psi ^{\prime \prime }-\sigma \psi ^{\prime }+\varrho =0,  \label{ODE}
\end{equation}%
with initial condition $\psi (0)=0$, whose solution is 
\begin{equation*}
\psi (s)=\frac{1}{\sigma }\left[ \varrho s-c_{1}(e^{\sigma s}-1)\right] .
\end{equation*}%
We require $0<\psi ^{\prime }(s)\leq 1$, i.e., 
\begin{equation*}
0<\frac{1}{\sigma }(\varrho -c_{1}\sigma e^{\sigma s})\leq 1,
\end{equation*}%
which implies 
\begin{equation*}
-\frac{\varrho }{\sigma }e^{-\sigma s}<-c_{1}\leq \left( \frac{\sigma
-\varrho }{\sigma }\right) e^{-\sigma s}.
\end{equation*}%
Let $\rho =R-r$. Then everything is in order if 
\begin{equation}
\frac{\varrho -\sigma }{\sigma }\leq c_{1}<\frac{\varrho }{\sigma }%
e^{-\sigma \rho }.  \label{c1bis}
\end{equation}%
This holds if 
\begin{equation}
\ln \left( \frac{\varrho }{\varrho -\sigma }\right) <\sigma \rho .
\label{lnd}
\end{equation}%
As $\frac{\varrho }{\varrho -\sigma }<1$, condition \eqref{lnd} is satisfied
if 
\begin{equation*}
\rho <\frac{1}{\sigma }\ln \left( \frac{\varrho }{\varrho -\sigma }\right) .
\end{equation*}%
Therefore, if 
\begin{equation}
\delta <\min \left\{ R-r,\frac{1}{\sigma }\ln \left( \frac{\varrho }{\varrho
-\sigma }\right) \right\} ,  \label{delta}
\end{equation}%
then, for 
\begin{equation*}
\psi (s)=\frac{1}{\sigma }\left[ \varrho s-c_{1}(e^{\sigma s}-1)\right]
,\quad s\in \lbrack 0,\delta ],
\end{equation*}%
with $c_{1}=\frac{\varrho -\sigma }{\sigma }$, we have $\psi (0)=0$, $0<\psi
^{\prime }(s)\leq 1$, and $\psi ^{\prime \prime }(s)=-c_{1}\sigma e^{\sigma
s}<0$. Furthermore, $\psi (s)\geq 0$ on $[0,\delta ]$. Finally, defining 
\begin{equation*}
\Lambda _{p,\delta }=B_{r+\delta }(p_{0})\cap \Omega ,
\end{equation*}%
we obtain the function 
\begin{equation*}
w(x)=\varphi (x)+\psi (d(x)),
\end{equation*}%
with 
\begin{equation}
\psi (s)=\frac{1}{\sigma }\left[ \varrho s+\left( \frac{\sigma -\varrho }{%
\sigma }\right) (e^{\sigma s}-1)\right] ,\quad 0\leq s\leq \delta ,
\label{psi}
\end{equation}%
which satisfies $w(p)=\varphi (p)$, $w(x)\geq \varphi (x)$ for all $x\in 
\overline{\Lambda }_{p,\delta }$, and $\mathcal{M}(w(x))\leq 0$ in $%
\overline{\Lambda }_{p,\delta }$.
\end{proof}

\subsection{Proofs of Theorem \protect\ref{TA} and Corollary \protect\ref%
{Cor1}}

\textbf{Proof of Theorem \ref{TA}.}

In the case where $\partial \Omega ^{-}=\emptyset $ and $\Omega $ is
bounded, the result follows, for instance, from Theorem 1 of \cite{ARS}.

Suppose $\Omega $ bounded and that $\partial \Omega ^{-}$ satisfies the
locally Hadamard $r$-es condition.

Given any $p\in \partial \Omega ^{-}$, consider the neighborhood $\Lambda
_{p,\delta }$ as defined in Proposition \ref{pb}, where $\delta $ satisfies
condition (\ref{delta}). Define the function 
\begin{equation*}
w_{p}=\varphi +\psi (d_{p}),
\end{equation*}%
where $\psi $ is as in (\ref{psi}). To construct a barrier from above at $p$%
, we require that 
\begin{equation*}
w_{p}(x)=\varphi (x)+\psi (\delta )\geq \sup_{\overline{\Omega }}\varphi
\quad \text{on }\partial \Lambda _{p,\delta }\setminus \partial \Omega ,
\end{equation*}%
which holds if 
\begin{equation*}
\psi (\delta )\geq \sup_{\overline{\Omega }}\varphi -\inf_{\overline{\Omega }%
}\varphi =:\omega .
\end{equation*}

Since $\delta $ is independent of the point $p\in \partial \Omega ^{-}$,
Proposition \ref{pb} implies that, if 
\begin{equation}
\omega \leq \frac{1}{\sigma }\left[ \varrho \delta +\left( \frac{\sigma
-\varrho }{\sigma }\right) \left( e^{\sigma \delta }-1\right) \right] ,
\label{osc}
\end{equation}%
(which is the expression (\ref{mo}) since that $\sigma =\varrho (1-\theta )$%
), then we can construct an upper barrier at each $p\in \partial \Omega ^{-}$
by setting 
\begin{equation*}
v_{p}(x)=%
\begin{cases}
\min \left\{ w_{p}(x),\inf_{\partial \Lambda _{p,\delta }\setminus \partial
\Omega }w_{p}\right\} & \text{if }x\in \Lambda _{p,\delta }, \\ 
\inf_{\partial \Lambda _{p,\delta }\setminus \partial \Omega }w_{p}, & \text{%
if }x\in \Omega \setminus \Lambda _{p,\delta }.%
\end{cases}%
\end{equation*}

For a lower barrier at $p\in \partial \Omega ^{-}$, just work with $\xi
_{p}:=\varphi -\psi \circ d$.

At points of $\partial \Omega$ where the domain is mean convex, we proceed
as in the proof of Theorem 1 of \cite{ARS}, and no additional geometric
assumptions are required to construct barriers.

The conclusion now follows from classical elliptic PDE theory \cite{GT}.

\bigskip If $\Omega $ is unbounded and $\tau <\infty $, the result follows
directly from the method of Perron (see Theorem 5 of \cite{RT}). Indeed,
define, at given $x\in \overline{\Omega },$ 
\begin{equation*}
u\left( x\right) =\sup \left\{ s\left( x\right) ;s\in C^{0}\left( \Omega
\right) \text{ is generalized subsolution of }\mathcal{M}\text{ and }%
s|_{\partial \Omega }\leq \varphi |_{\partial \Omega }\right\} .
\end{equation*}%
Since every subsolution used to define $u$ is bounded above by $\sup_{\Omega
}\varphi $, the function $u$ is well-defined. Moreover, by Perron's method,
we have $u\in \nolinebreak C^{\infty }\left( \Omega \right) $ and $\mathcal{M%
}\left( u\right) =0.$ The barriers at the boundary ensures that $u\in
C^{2}\left( \overline{\Omega }\right) $ and $u|_{\partial \Omega }=\varphi .$

\textbf{Proof of Corollary \ref{Cor1}.}

When $M$ is a Hadamard manifold, we have $R=\infty $, so we can take $\delta 
$ in (\ref{delta}) as 
\begin{equation*}
\delta =\frac{1}{\sigma }\ln \left( \frac{\varrho }{\varrho -\sigma }\right)
.
\end{equation*}%
Substituting this into (\ref{osc}) and recalling that $\sigma =\varrho
(1-\theta )$, condition (\ref{osc}) becomes 
\begin{equation}
\omega \leq \frac{1}{\varrho (\theta -1)^{2}}\left( \theta -\ln \theta
-1\right) .  \label{HC}
\end{equation}%
Since $1<a<\theta $ and observing that the function $f(x)=x-\ln x-1$ is
positive and increasing for $x>1$, we have that 
\begin{equation}
\omega \leq \frac{a-\ln a-1}{\varrho (\theta -1)^{2}}  \label{HC2}
\end{equation}%
implies (\ref{HC}). Now observe that 
\begin{equation*}
\theta =a+b(n-1)\lambda _{r}+c\mu _{r}.
\end{equation*}%
Plugging this into (\ref{HC2}) leads to the equivalent inequality 
\begin{equation}
\frac{\sqrt{a-\ln a-1}-\sqrt{\varrho \omega }(a+c\mu _{r}-1)}{\sqrt{\varrho
\omega }\,b(n-1)}\leq \lambda _{r}.  \label{Fl}
\end{equation}

\begin{remark}
Since $\lambda_r < 0$, the left-hand side of (\ref{Fl}) must be negative.
This occurs if 
\begin{equation*}
\omega < \frac{a - \ln a - 1}{\varrho (a + c\mu_r - 1)^2},
\end{equation*}
which is guaranteed by condition (\ref{HC2}).
\end{remark}

\bigskip

\begin{example}
We now present an example illustrating that our condition ensures the
existence of a solution to the Dirichlet problem (\ref{DP}) even in cases
where the Jenkins-Serrin criterion fails to do so.

Let $\Omega \subset \mathbb{R}^{2}$ be the bounded domain whose boundary is
parametrized by 
\begin{equation*}
\gamma (t)=\xi (t)(\cos t,\sin t),\quad t\in \lbrack 0,2\pi ],
\end{equation*}%
where 
\begin{equation*}
\xi (t)=\sqrt{4\cos (2t)+\sqrt{16\cos ^{2}(2t)+12}}.
\end{equation*}%
Let $\varphi \in C^{2}(\overline{\Omega })$ be given by 
\begin{equation*}
\varphi (x,y)=\frac{e^{-y}}{20}+1.
\end{equation*}

\begin{figure}[h]
\centering
\IfFileExists{DomE.jpg}{\fbox{\includegraphics[scale=0.75]{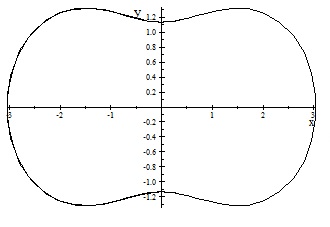}}}{%
\fbox{\rule{5cm}{3cm}\ \text{[DomE.jpg missing]}}}
\caption{The domain used in the example}
\end{figure}

Regarding the geometry of $\partial \Omega $, we compute 
\begin{equation*}
-0.45\approx \lambda _{r}\leq k_{\partial \Omega }\leq k:=\sup_{\partial
\Omega }|k_{\partial \Omega }|\approx 0.82,
\end{equation*}%
so that $\mu _{r}=-\lambda _{r}$. For the boundary data $\varphi |_{\partial
\Omega }$, we find 
\begin{equation*}
\omega _{\varphi }(\overline{\Omega })=\omega _{\varphi }(\partial \Omega
)\approx 0.17,
\end{equation*}%
with $\tau \approx 0.1871$, $\theta \approx 9.05$, and the maximal
admissible value of $\delta $ from (\ref{delta}) is approximately $1.37$.
Choosing $\delta =1$ for simplicity, the right-hand side of inequality (\ref%
{mo}) becomes approximately $0.43$. Since $\omega \approx 0.17<0.43$, the
condition is satisfied.

On the other hand, following the Jenkins-Serrin method, we proceed as in
case I) above for computing $B$. Using the corresponding notation and in
order to obtain the maximum value for $B$, we estimate: 
\begin{equation*}
l=0.25,\quad l^{\prime }\approx 0.1767,\quad \frac{\pi }{l^{\prime }}\approx
17.7792,\quad k\approx 0.82,
\end{equation*}%
\begin{equation*}
|D^{2}\varphi |\approx 0.1871,\quad A\approx 17.7792,\quad C\approx
10.2526,\quad \mathcal{H}\approx 0.45.
\end{equation*}%
These values yield 
\begin{equation*}
B\approx 0.0072<\omega _{\varphi }(\partial \Omega ),
\end{equation*}%
so the Jenkins-Serrin condition fails, while our condition still guarantees
ex\nolinebreak istence.
\end{example}

\section{Mean Curvature Flow of Graphs without Mean Convexity}

\label{SecFlow}

The differential inequality leading to the ordinary differential equation 
\begin{equation}
\psi ^{\prime \prime }-\sigma \,\psi ^{\prime }+\varrho =0,
\label{ODEreprise}
\end{equation}%
introduced in the previous section, naturally reappears in the study of the
mean curvature flow of graphs. The same function $\psi $ that provides
elliptic barriers for the Dirichlet problem also yields parabolic barriers
that allow one to start and control the evolution even when $\partial \Omega 
$ is not mean convex. Thus equation \eqref{ODEreprise} plays a central role
in connecting the elliptic and parabolic regimes.

In the elliptic construction, the auxiliary function is written in terms of
the distance $d_{p}(x)=\mathrm{dist}(x,\Sigma _{p})$ to the exterior
geodesic sphere $\Sigma _{p}$ associated with each point $p\in \partial
\Omega ^{-}$, since these spheres encode the geometric $r$--exterior sphere
condition.

In the parabolic setting, in order to obtain barriers that are compatible
with the Dirichlet condition along the whole boundary, it is more natural to
work with the distance to the boundary, 
\begin{equation*}
d(x)=\mathrm{dist}(x,\partial \Omega ).
\end{equation*}%
Although $d(x)$ and $d_{p}(x)$ need not to coincide as functions in a
neighborhood of $p$, along $\partial \Omega $ the two functions have the
same first and second derivatives in the normal direction at $p$, that is,
their first and second jets at $p$ coincide. Since all the differential
inequalities leading to the ODE~\eqref{ODEreprise} depend only on the values
of $\nabla d$ and $\nabla ^{2}d$ along the boundary, this second--order
agreement is sufficient to transfer the elliptic barrier construction to the
parabolic framework without any modification of the resulting ODE.

Let $\Omega \subset M^{n}$ be a $C^{2}$ domain in a complete Riemannian
manifold and let $\varphi \in C^{2}(\overline{\Omega })$ satisfy the
structural condition of Theorem \ref{TA} (or, in the Hadamard case,
condition (\ref{SH})). We conclude from the above arguments that there is a
tubular neighborhood of $\partial \Omega $, 
\begin{equation*}
U_{\delta }:=\{x\in \Omega :d(x)<\delta \},
\end{equation*}%
where the signed distance function $d$ to $\partial \Omega $ is of class $%
C^{2}$ and, chosen $\psi $ as in Section 1.1 (Proposition \ref{pb}), 
\begin{equation*}
w^{\pm }(x):=\varphi (x)\pm \psi (d(x)),\qquad x\in U_{\delta },
\end{equation*}%
are such that $w^{-}$ is a subsolution and $w^{+}$ is a supersolution for
the operator $\mathcal{M}$ in $U_{\delta }$.

\begin{theorem}
Let $\Omega \subset M^{n}$ be a $C^{2}$ domain in a complete Riemannian
manifold, and let $\varphi \in C^{2}(\overline{\Omega })$ satisfy the
structural condition of Theorem \ref{TA} (or, in the Hadamard case,
condition (\ref{SH})). Let $u_{0}\in C^{2}(\overline{\Omega })$ satisfy 
\begin{equation*}
u_{0}|_{\partial \Omega }=\varphi ,~w^{-}(x)\leq u_{0}(x)\leq w^{+}(x)\text{
for all }x\in U_{\delta },\text{ }\left\vert u_{0}\right\vert \leq
\inf_{\partial U_{\delta }\backslash \partial \Omega }w^{+}\text{,}
\end{equation*}%
where $w^{\pm }$ and $U_{\delta }$ are as defined above. Then there exists $%
T>0$, depending only on $\varphi $ (through $\tau (\varphi )$) and on the
geometric parameters $(\lambda _{r},\mu _{r})$ arising in the construction
of the boundary barriers \emph{and, in particular, independent of the choice
of $u_{0}$}, such that the initial--boundary value problem 
\begin{equation}
\begin{cases}
\partial _{t}u=\mathcal{M}(u), & (x,t)\in \Omega \times (0,T), \\ 
u(x,0)=u_{0}(x), & x\in \Omega , \\ 
u(x,t)=\varphi (x), & (x,t)\in \partial \Omega \times \lbrack 0,T),%
\end{cases}
\label{MCF2}
\end{equation}%
admits a classical solution 
\begin{equation*}
u\in C^{2,1}(\overline{\Omega }\times \lbrack 0,T]).
\end{equation*}%
Moreover, for every $t\in \lbrack 0,T]$, the hypersurface 
\begin{equation*}
\Gamma _{t}=\{(x,u(x,t)):x\in \Omega \}
\end{equation*}%
remains a graph over $\Omega $, even if $\partial \Omega $ is not mean
convex.
\end{theorem}

\begin{proof}
The proof relies directly on the equation~\eqref{ODEreprise}. Let $d(x)$
denote the signed distance to $\partial \Omega $ and consider $w^{\pm }$ and 
$U_{\delta }$ are as defined above.

Since $w^{\pm }$ are independent of $t$, we have $\partial _{t}w^{\pm }=0$,
and because $\psi $ satisfies~\eqref{ODEreprise}, the functions $w^{+}$ and $%
w^{-}$ verify 
\begin{equation*}
\partial _{t}w^{+}-\mathcal{M}(w^{+})\geq 0,\qquad \partial _{t}w^{-}-%
\mathcal{M}(w^{-})\leq 0
\end{equation*}%
in the tubular neighborhood $U_{\delta }$ of $\partial \Omega $, and both
coincide with $\varphi $ on $\partial \Omega $.

Under the quantitative restriction~\eqref{mo} (or~\eqref{SH}), we have $%
w^{-}(x)\leq \varphi (x)\leq w^{+}(x)$ on $\partial \Omega $ and $w^{-}\leq
w^{+}$ throughout $U_{\delta }$. To apply the parabolic maximum principle
globally, we extend $w^{\pm }$ from $U_{\delta }$ to whole $\Omega $ by a
standard $C^{2}$ gluing argument, by smoothing the capped function 
\begin{equation*}
W^{+}\left( x\right) =\left\{ 
\begin{array}{c}
\min \{w^{+}\left( x\right) ,\inf_{\partial \Omega _{\delta }\backslash
\partial \Omega }w^{+}\}\text{ se }x\in \overline{\Omega }_{\delta } \\ 
\inf_{\partial \Omega _{\delta }\backslash \partial \Omega }w^{+}\text{ se }%
x\in \Omega \backslash \Omega _{\delta }%
\end{array}%
\right.
\end{equation*}%
(which is a $C^{0}$ supersolution in $\overline{\Omega }$ for the operator $%
\mathcal{M}$) in such a way that the supersolution property is preserved,
and we call it again $W^{+}$ (analogous construction for $W^{-}$). From our
hypothesis on $u_{0}$, it follows that $W^{-}\leq u_{0}\leq W^{+}$ on $%
\overline{\Omega }$ with $u_{0}=\varphi $ on $\partial \Omega $. The
classical theory for quasilinear parabolic equations (see \cite{L})
guarantees local existence and uniqueness provided the structure is
uniformly parabolic. We thus verify this uniform parabolicity within the
region bounded by $W^{-}$ and $W^{+}$.

Since $\psi $ is smooth and monotone, the derivatives $|\nabla w^{\pm }|$
are bounded by constants depending only on $\tau (\varphi )$ and $(\lambda
_{r},\mu _{r})$. By the parabolic maximum principle, the solution $u$ of~%
\eqref{MCF2} remains confined between $W^{-}$ and $W^{+}$, so that 
\begin{equation*}
W^{-}(x)\leq u(x,t)\leq W^{+}(x)\quad \text{for all }(x,t)\in \overline{%
\Omega }\times \lbrack 0,T].
\end{equation*}%
Hence $|\nabla u|$ is uniformly bounded because $\psi ^{\prime }(d)$ is
bounded, and therefore the coefficients 
\begin{equation*}
a^{ij}(\nabla u)=\delta ^{ij}(1+|\nabla u|^{2})-u_{i}u_{j}
\end{equation*}%
remain positive definite. The equation~\eqref{MCF2} is thus uniformly
parabolic during the time interval considered.

With this structure established, the fixed--point argument applies: one
linearizes \eqref{MCF2} around $u_{0}$, solves the linearized problem via
the Schauder theory, and iterates. The limit yields a solution $u\in C^{2,1}$
on $[0,T]$, where $T$ depends only on the upper bound of $|\nabla u|$
obtained from the barriers, hence only on $\tau (\varphi )$, $\lambda _{r}$
and $\mu _{r}$.

The boundary control provided by $\psi $ replaces the classical mean
convexity assumption. Since the gradient is uniformly bounded, the graph
cannot develop vertical tangents and remains graphical for all $t\in \lbrack
0,T]$. Interior gradient and curvature estimates, identical to those derived
by Ecker--Huisken~\cite{EH}, ensure that no curvature blow--up occurs before 
$T$. Finally, $u(x,t)\rightarrow u_{0}(x)$ as $t\rightarrow 0$ by parabolic
regularity and compatibility of the initial and boundary data.

Consequently, the function $\psi $ solving the ODE~\eqref{ODEreprise}
provides a universal stationary parabolic barrier that replaces the mean
convexity of the boundary, establishing local--in--time existence of the
mean curvature flow of graphs over arbitrary domains.
\end{proof}

\begin{remark}
Let $[0,U)$ denote the maximal time interval of existence of the classical
solution $u$ to \eqref{MCF2}. By the interior estimates of Ecker--Huisken,
for any sequence $t_{k}\nearrow U$ there exists a subsequence (still denoted 
$t_{k}$) such that $u(\cdot ,t_{k})$ converges in $C_{\mathrm{loc}}^{\infty
}(\Omega )$ to a function $v\in C^{\infty }(\Omega )$. Moreover, $v$
satisfies the stationary minimal surface equation $\mathcal{M}(v)=0$ in $%
\Omega $, hence defines a minimal graph in the interior of the domain.
However, since all these estimates are purely interior, no information can
be extracted in general about the behavior of the solution at the boundary
as $t\nearrow U$. In particular, the possible obstruction to extending the
flow beyond $U$ is entirely of boundary nature.
\end{remark}

\bigskip

{\small Ari J. Aiolfi}

{\small Department of Mathematics/CCNE - Federal University of Santa Maria}

{\small Santa Maria RS - Brazil (ari.aiolfi@ufsm.br)}

{\small \bigskip }

{\small Giovanni da Silva Nunes \& Lisandra Sauer}

{\small Institute of Physics and Mathematics - Federal University of Pelotas}

{\small Pelotas RS - Brazil (giovanni.nunes@ufpel.edu.br,
lisandra.sauer@ufpel.edu.br)}

{\small \bigskip }

{\small Jaime B. Ripoll}

{\small Department of Mathematics/IME - University of Sao Paulo}

{\small Department of Mathematics/IME - Federal University of Rio Grande do Sul}

{\small Porto Alegre RS - Brazil (jaime.ripoll@ufrgs.br)}

{\small \bigskip }

{\small Rodrigo Soares}

{\small Institute of Mathematics and Statistics - Federal University of Rio
Grande}

{\small Rio Grande RS - Brazil (rodrigosoares@furg.br)}


\begin{thebibliography}{99}
\bibitem{AGLR} A. Aiolfi, G. Nunes, L. Sauer, R. Soares: Compact CMC graph
in $M\times \mathbb{R}$ with boundary in two horizontal slices, Bull. Braz.
Math. Soc. - New Series, \textbf{49}, 659--672 (2018).

\bibitem{ANSS} A. Aiolfi, G. Nunes, L. Sauer, R. Soares: The Dirichlet
problem for the minimal hypersurface equation with Lipschitz continuous
boundary data in a Riemannian manifold, Pacific J. of Math, (1) \textbf{307}%
, 1-12 (2020).

\bibitem{ARS} A. Aiolfi, J. Ripoll, M. Soret: The Dirichlet problem for the
minimal hypersurface equation on arbitrary domains of a Riemannian manifold,
Manuscripta Math., \textbf{149}, 71--81 (2016).

\bibitem{AW} S. Altschuler, L.-F. Wu: Translating surfaces of the mean
curvature flow, J. Differential Geom., \textbf{48}, 475--510 (1998).

\bibitem{DHL} Dajczer, M., Hinojosa, P., de Lira, J.H.: Killing graphs with
prescribed mean curvature. Calc. Var. Partial Differ. Equation., \textbf{33}%
, 231--248 (2008).

\bibitem{EH2} K. Ecker, G. Huisken, Mean curvature evolution of entire
graphs, Ann. of Math., (2) \textbf{130}, 453--471 (1989).

\bibitem{EH} K. Ecker, G. Huisken: Interior estimates for hypersurfaces
moving by mean curvature, Invent. Math., \textbf{105}, 547--569 (1991).

\bibitem{ER} N. Espirito-Santo, J. Ripoll.: Some existence and nonexistence
theorems for compact graphs of constant mean curvature with boundary in
parallel planes. J. Geom. Anal., (4) \textbf{11}, 601--617 (2001).

\bibitem{F} R. Finn: On equations of minimal surface type, Annals of
Mathematics (2) \textbf{60}, 397--416 (1954).

\bibitem{GT} D. Gilbarg, N. Trudinger: Elliptic Partial Differential
Equations of Second Order. Springer-Verlag Berlin (1983).

\bibitem{JS} H. Jenkins, J. Serrin: The Dirichlet problem for the minimal
surface equation in higher dimensions. J. Reine Angew. Math., \textbf{229},
170--187 (1968).

\bibitem{KT} N. Kutev, F. Tomi: Existence and nonexistence for the exterior
Dirichlet problem for the minimal surface equation in the plane.
Differential Integral Equations (6) \textbf{11}, 917--928 (1998).

\bibitem{L} G. M. Lieberman: Second Order Parabolic Differential Equations,
World Scientific, Singapore (1996).

\bibitem{RS} J. Ripoll, L. Sauer: A Note on the Dirichlet problem for the
Minimal Surface Equation in Nonconvex Planar Domains. Matem. Contemp., 
\textbf{35}, 177-183 (2008).

\bibitem{RT} J.Ripoll, M. Telichevesky: On the Asymptotic Plateau Problem
for CMC Hypersurfaces in Hyperbolic Space. Bull. Braz. Math. Soc. - New
Series, \textbf{50,} 575--585 (2019),

\bibitem{W} G. H. Williams: The Dirichlet problem for the minimal surface
equation with Lipschitz continuous boundary data. J. Reine Angew. Math., 
\textbf{354}, 123--140 (1984).
\end{thebibliography}
\end{document}